\documentclass[11pt,reqno]{amsart}
\usepackage{mathrsfs}

\font\bbbld=msbm10 scaled\magstephalf

\usepackage{hyperref}

\usepackage{graphicx}

\usepackage{xcolor}

\newcommand{\bpartial}{\bar{\partial}}

\def \a{\alpha}

\def \p{\partial}

\newcommand{\bfR}{\hbox{\bbbld R}}

\newtheorem{theorem}{Theorem}[section]
\newtheorem{lemma}[theorem]{Lemma}

\newtheorem{corollary}[theorem]{Corollary}

 \theoremstyle{definition}

\theoremstyle{remark}
\newtheorem{remark}[theorem]{Remark}

\numberwithin{equation}{section}

\begin{document}
\setlength{\baselineskip}{1.2\baselineskip}

\title[Fully Nonlinear Elliptic Equations: $L^\infty$ estimate]
{On a class of fully nonlinear elliptic equations
    on closed Hermitian manifolds II: $L^\infty$ estimate}

\author{Wei Sun}

\address{Department of Mathematics, Ohio State University,
         Columbus, OH 43210}
\email{sun.220@osu.edu}

\begin{abstract}
We study a class of fully nonlinear elliptic equations on closed Hermitian manifolds. Under the assumption of cone condition, we derive the $L^\infty$ estimate directly.
\end{abstract}

\maketitle

\section{Introduction}
\label{ma2-int}
\setcounter{equation}{0}
\medskip

Let $(M, \omega)$ be a compact Hermitian manifold of complex dimension $n\geq 2$ and $\chi$ a smooth real $(1,1)$ form on $M$. For convenience, we shall write
\begin{equation*}
	\omega = \sqrt{-1} \sum_{i,j} g_{i\bar j} dz^i \wedge d\bar z^j
\end{equation*}
and
\begin{equation*}
	\chi = \sqrt{-1} \sum_{i,j} \chi_{i\bar j} dz^i \wedge d\bar z^j
\end{equation*}
respectively in any local coordinate chart.

Throughout this paper we will use the shorthand $\chi_u = \chi + \sqrt{-1} \p\bpartial u$. We are concerned with the following two types of elliptic equations:
\begin{enumerate}
\item {\em The complex $k$-Hessian equation.} For $2 \leq k \leq n$ and $\chi \in \Gamma^k_\omega$,
\begin{equation}
\label{hessian-equation}
	\chi^k_u \wedge \omega^{n - k} = \psi \omega^n ,\qquad \text{with } \chi_u \in \Gamma^k_\omega. 
\end{equation}

\item {\em The complex $(k,l)$-quotient equation.} For $ 1 \leq l < k \leq n$ and $\chi \in \Gamma^k_\omega$,
\begin{equation}
\label{quotient-equation}
	\chi^k_u \wedge \omega^{n - k} = \psi \chi^l_u \wedge \omega^{n - l} ,\qquad \text{with } \chi_u \in \Gamma^k_\omega.
\end{equation}
Following \cite{SW08}, \cite{FLM11} and \cite{GSun12}, we define for $\psi \in C^0(M)$, $\psi > 0$
\begin{equation}
\label{cone-condition}
	\mathscr{C}_{k,l} (\psi) := \{[\chi]\, : \exists \chi' \in \Gamma^k_\omega \cap [\chi],\, k \chi'^{k - 1} \wedge \omega^{n - k} > l \psi \chi'^{l - 1} \wedge \omega^{n - l}\} .
\end{equation}
If $[\chi] \in \mathscr{C}_{k,l} (\psi)$, we say that $\chi$ satisfies the cone condition for equation~\eqref{quotient-equation} with respect to $\psi$.

\end{enumerate}
Here $\psi$ is a smooth positive function on $M$, and $\Gamma^k_\omega$ is the set of all the real $(1, 1)$ forms whose eigenvalue set with respect to $\omega$ belong to $k$-positive cone in $\bfR^n$.

These equations include some of the most important partial differential equations in complex geometry and analysis.  The $n$-Hessian equation corresponds to the complex Monge-Amp\`ere equations which plays central roles in K\"ahler geometry as well as problems outside K\"ahler geometry since the famous work of Yau~\cite{Yau78} (see also Aubin~\cite{Aubin78}), while the $(n,n-1)$-quotient equations appears in a problem proposed by Donaldson~\cite{Donaldson99a} in the setting of moment maps and another by Chen~\cite{Chen00b} in  the study of Mabuchi energy.

The study in this paper reveals the key role of the cone condition in complex geometric equations on closed manifolds. Similar to the subsolution condition for the Dirichlet problem, the cone condition is very likely to help us to remove some geometric assumptions, e.g. positive curvatures. Unlike the subsolution for the Dirichlet problem,  in previous works the cone condition  can  be essentially used to discover $C^2$ estimate only. To the best of the author's knowledge, this is the first time to derive such $L^\infty$ bound directly from the cone condition. Indeed, the $L^\infty$ bound is often the most difficult part in solving elliptic or parabolic problems on closed manifolds.

The following theorem states the result on closed K\"ahler manifolds.
\begin{theorem}
\label{main-theorem}
Let $(M,g)$ be a closed K\"ahler manifold of complex dimension $n \geq 2$ and $\chi$ a smooth closed real $(1,1)$ form. Assume that $u\in C^2(M)$ satisfies either \eqref{hessian-equation} or \eqref{quotient-equation}. Then there is a uniform $C^0$ a priori estimate for $u$ depending only on $(M, \omega)$, $\chi$ and $\psi$.
\end{theorem}
\begin{remark}
	It is easy to see that the complex $k$-Hessian equation can be treated as a particular case of the complex quotient equations when $l = 0$.  However, the complex Hessian equations have a natural strong cone condition, that is $\mathscr{C}_{k,0} (\psi) = \{[\chi]\, :\, \Gamma^k_\omega \cap [\chi] \neq \emptyset\}$. Furthermore, $[\chi] \in \mathscr{C}_{k,0} (\psi)$ for any smooth positive function $\psi$.
\end{remark}

On Hermitian manifolds, the equations are much more difficult to treat due to the torsion terms. In this paper we shall only study the complex Monge-Amp\`ere type equations on closed Hermitian manifolds, which is exactly the $(n, n - \a)$-quotient equations 
\begin{equation}
\label{ma2-main-equation}
	\chi^n_u = \psi \chi^{n - \alpha}_u \wedge \omega^\alpha, \qquad \text{with } \chi_u > 0
\end{equation}
where $\psi$ is a smooth positive function  and $1 \leq \alpha \leq n$.

Our main result on Hermitian manifolds is the following  {\em a priori} estimates.
\begin{theorem}
\label{ma2-theorem-estimate}
Let $(M,\omega)$ be a closed Hermitian manifold of complex dimension $n \geq 2$ and $u$ be a smooth admissible solution to equaion \eqref{ma2-main-equation}. Suppose that $[\chi] \in \mathscr{C}_{n, n -\a} (\psi)$. Then there are uniform $C^\infty$ a priori estimates for $u$.
\end{theorem}

\begin{remark}
By the proofs, it is straightforward to verify the uniqueness, up to a constant, of the solutions to equations~\eqref{hessian-equation},~\eqref{quotient-equation} on closed K\"ahler manifolds and equation~\eqref{ma2-main-equation} on closed Hermitian manifolds.
\end{remark}

One significant application is the regularity and existence of the solution to equation~\eqref{ma2-main-equation} after rescaling on closed Hermitian manifolds.

\begin{corollary}
\label{ma2-theorem-solution}
Let $(M^n,\omega)$ be a closed Hermitian manifold of complex dimension $n$ and $\chi$ a smooth Hermitian metric on $M^n$. Suppose that $\chi$ satisfies the cone condition with respect to $\psi$. Then  there exists a unique solution to equation \eqref{ma2-main-equation} up to a constant multiple if one of the following conditions holds true:
\begin{enumerate}

\item $\a = n$ ;

\item $\chi$ and $\omega$ satisfy
	\begin{equation*}
	\label{ma2-theorem-hermitian-condition}
    		\frac{\chi^n}{\chi^{n - \a} \wedge \omega^\a}\leq \psi ; 
	\end{equation*}

\item $\chi$ and $\omega$ are both K\"ahler, and $\psi \geq c$ where
	\begin{equation*}
		\label{ma2-kahler-constant}
		c = \frac{\int_M \chi^n}{\int_M \chi^{n - \a} \wedge \omega^\a}.
	\end{equation*}

\end{enumerate}	
\end{corollary}
\begin{remark}
The first case was achieved by Tosatti and Weinkove~\cite{TWv10a},~\cite{TWv10b}, and can be treated as a particular case of the second but with a strong cone condition.
\end{remark}

\bigskip

\section{The estimates in K\"ahler geometry}
\label{kahler}
\setcounter{equation}{0}
\medskip

According to Tosatti and Weinkove~\cite{TWv10a},~\cite{TWv10b}, it suffices to show 
\begin{equation}
\label{ma2_2_main_inequality}
	\int_M |\p e^{- \frac{p}{2} u}|^2_g \omega^n \leq C p\int_M e^{- p u} \omega^n 
\end{equation}
for $p$ large enough. We refer the readers to \cite{TWv10a},~\cite{TWv10b} and~\cite{Yau78} for more details.

\begin{lemma} 
\label{hessian-lemma-estimate}
Let $u$ be a smooth admissible solution of $k$-Hessian equation~\eqref{hessian-equation}. Then there are uniform constants $C$, $p_0$ such that for all $p \geq p_0$, inequality~\eqref{ma2_2_main_inequality} holds true.
\end{lemma}
\begin{proof}

 It is easy to see that we have the following pointwise equality
\begin{equation}
\label{hessian-lemma-estimate-1}
	\chi^k_u \wedge \omega^{n - k} - \chi^k \wedge \omega^{n - k} = k \int^1_0 \sqrt{-1} \p\bpartial u \wedge \chi^{k - 1}_{t u} \wedge \omega^{n - k} dt .
\end{equation}
By the concavity of $S_k$, we know that if both $\chi$ and $\chi_u$ are in $\Gamma^k_\omega$, so is $\chi_{t u}$.

We compute directly,
\begin{equation}
\label{hessian-lemma-estimate-2}
\begin{aligned}
	C \int_M e^{- p u} \omega^n &\geq \int_M e^{- p u} (\chi^k_u \wedge \omega^{n - k} - \chi^k \wedge \omega^{n - k}) \\
	&= k p \int_M e^{- p u} \left(\int^1_0 \sqrt{-1} \p u\wedge \bpartial u \wedge \chi^{k - 1}_{t u} \wedge \omega^{n - k} dt\right).
\end{aligned}
\end{equation}
By the concavity, we have the following elementary pointwise inequality for $ 1\leq i \leq n$
\begin{equation}
\label{hessian-lemma-estimate-3}
	S^{\frac{1}{k - 1}}_{k - 1;i} (\chi_{t u}) \geq (1 - t) S^{\frac{1}{k - 1}}_{k - 1;i} (\chi) + t S^{\frac{1}{k - 1}}_{k - 1;i}  (\chi_u), 
\end{equation}
and hence
\begin{equation}
\label{hessian-lemma-estimate-4}
	\sqrt{-1} \p u \wedge \bpartial u \wedge \chi^{k - 1}_{t u} \wedge \omega^{n - k} \geq (1 - t)^{k - 1} \sqrt{-1} \p u \wedge \bpartial u \wedge \chi^{k - 1} \wedge \omega^{n - k} .
\end{equation}
Therefore
\begin{equation}
\label{hessian-lemma-estimate-5}
\begin{aligned}
	C \int_M e^{- p u} \omega^n &\geq k p \int_M e^{- p u} \left(\int^1_0 (1 - t)^{k - 1} \sqrt{-1} \p u \wedge \bpartial u \wedge \chi^{k - 1} \wedge \omega^{n - k}\right) \\
	&\geq p \int_M e^{- p u} \sqrt{-1} \p u \wedge \bpartial u \wedge \chi^{k - 1} \wedge \omega^{n - k} \\
	&= \frac{4 c_0}{n p} \int_M |\p e^{-\frac{p}{2} u} |^2_g \omega^n 
\end{aligned}
\end{equation}
for some uniform constant $c_0 > 0$.
\end{proof}

Albeit the complex Hessian equations are particular cases of quotient equations, the zero estimate is much easier because of the implicit strong cone condition. For complex quotient equations, we can explicitly see the role of cone condition. Nevertheless, we remark that the $L^\infty$ bound on the solution to Hessian equations depends on $\sup_M \psi$ while that on the solution to quotient equations depends on the derivatives of $\psi$ up to second order.

\begin{lemma}
\label{quotient-lemma-estimate}
Let $u$ be a smooth admissible solution of $(k,l)$-quotient equation~\eqref{quotient-equation}. Then there are uniform constants $C$, $p_0$ such that for all $p \geq p_0$, inequality~\eqref{ma2_2_main_inequality} holds true.
\end{lemma}
\begin{proof}
Without loss of generality, we can assume
\begin{equation}
\label{quotient-lemma-estimate-1}
	k \chi^{k - 1} \wedge \omega^{n - k} > l \psi \chi^{l - 1} \wedge \omega^{n - l}.
\end{equation}
Also, by the monotony of $S_k/S_l$, we have
\begin{equation}
\label{quotient-lemma-estimate-2}
	k \chi^{k - 1}_u \wedge \omega^{n - k} > l \psi \chi^{l - 1}_u \wedge \omega^{n - l} .
\end{equation}

We consider
\begin{equation}
\label{quotient-lemma-estimate-3}
	\int_M e^{- p u} \left((\chi^k_u \wedge \omega^{n - k} - \chi^k \wedge \omega^{n - k} ) - \psi (\chi^l_u \wedge \omega^{n - l} - \chi^l \wedge \omega^{n - l})\right) .
\end{equation}
On one hand,
\begin{equation}
\label{quotient-lemma-estimate-4}
\begin{aligned}
	&\, \int_M e^{- p u} \left((\chi^k_u \wedge \omega^{n - k} - \chi^k \wedge \omega^{n - k} ) - \psi (\chi^l_u \wedge \omega^{n - l} - \chi^l \wedge \omega^{n - l})\right) \\
	=&\, \int_M e^{- p u} \left(\frac{\chi^k_u \wedge \omega^{n - k}}{\chi^l_u \wedge \omega^{n - l}} - \frac{\chi^k \wedge\omega^{n - k}}{\chi^l \wedge \omega^{n - l}}\right) \chi^l \wedge \omega^{n - l} \\
	\leq&\, C \int_M e^{- p u} \chi^l \wedge \omega^{n - l} .
\end{aligned}
\end{equation}
On the other hand, we have the pointwise equality
\begin{equation}
\label{quotient-lemma-estimate-5}
\begin{aligned}
	&\, (\chi^k_u \wedge \omega^{n - k} - \chi^k \wedge \omega^{n - k} ) - \psi (\chi^l_u \wedge \omega^{n - l} - \chi^l \wedge \omega^{n - l}) \\
	=& \int^1_0 \sqrt{-1} \p\bpartial u \wedge \left(k \chi^{k - 1}_{t u} \wedge \omega^{n - k} - l \psi \chi^{l - 1}_{t u} \wedge \omega^{n - l}\right) dt 
\end{aligned}
\end{equation}
and hence
\begin{equation}
\label{quotient-lemma-estimate-6}
\begin{aligned}
	&\, \int_M e^{- p u} \left((\chi^k_u \wedge \omega^{n - k} - \chi^k \wedge \omega^{n - k} ) - \psi (\chi^l_u \wedge \omega^{n - l} - \chi^l \wedge \omega^{n - l})\right) \\
	=&\, p \int^1_0 \left(\int_M e^{- p u} \sqrt{-1} \p u \wedge \bpartial u \wedge \left(k \chi^{k - 1}_{t u} \wedge \omega^{n - k} - l \psi \chi^{l - 1}_{t u} \wedge \omega^{n - l}\right) \right) dt \\
	&\, - \frac{l}{p} \int^1_0 \left(\int_M e^{- p u} \sqrt{-1} \bpartial\p \psi \wedge \chi^{l - 1}_{t u} \wedge \omega^{n - l} \right) dt \\
	\geq&\, p \int^1_0 \left(\int_M e^{- p u} \sqrt{-1} \p u \wedge \bpartial u \wedge \left(k \chi^{k - 1}_{t u} \wedge \omega^{n - k} - l \psi \chi^{l - 1}_{t u} \wedge \omega^{n - l}\right) \right) dt \\
	&\, - \frac{C}{p} \int^1_0 \left(\int_M e^{- p u} \chi^{l - 1}_{t u} \wedge \omega^{n - l + 1} \right) dt .
\end{aligned}
\end{equation}

By the concavity of hyperbolic polynomials, for $0 < \zeta < 1$, 
\begin{equation}
\label{quotient-lemma-estimate-7}
	\frac{1}{\zeta} S^{\frac{1}{m}}_m(\chi_{\zeta t u}) + \left(1 - \frac{1}{\zeta}\right) S^{\frac{1}{m}}_m (\chi) \geq S^{\frac{1}{m}}_m (\chi_{tu})
\end{equation}
and thus
\begin{equation}
\label{quotient-lemma-estimate-8}
	S_m (\chi_{\zeta t u}) \geq \zeta^m S_m (\chi_{t u}) .
\end{equation}
This implies that,
\begin{equation}
\label{quotient-lemma-estimate-9}
\begin{aligned}
	&\, \int^1_0 \left( \int_M e^{- p u} \chi^{l - 1}_{t u} \wedge \omega^{n - l + 1}\right) dt \\
	\leq&\, 2^{l - 1} \int^1_0 \left( \int_M e^{- p u} \chi^{l - 1}_{\frac{t u}{2}} \wedge \omega^{n - l + 1}\right) dt \\
	\leq&\, 2^l \int^{\frac{1}{2}}_0 \left(\int_M e^{- p u} \chi^{l - 1}_{t u} \wedge \omega^{n - l + 1}\right) dt .
\end{aligned}
\end{equation}
Then we obtain, from \eqref{quotient-lemma-estimate-4}, \eqref{quotient-lemma-estimate-6} and \eqref{quotient-lemma-estimate-9},
\begin{equation}
\label{quotient-lemma-estimate-10}
\begin{aligned}
	 &\, p \int^1_0 \left(\int_M e^{- p u} \sqrt{-1} \p u \wedge \bpartial u \wedge \left(k \chi^{k - 1}_{t u} \wedge \omega^{n - k} - l \psi \chi^{l - 1}_{t u} \wedge \omega^{n - l}\right) \right) dt \\
	\leq&\, \frac{C}{p} \int^{\frac{1}{2}}_0 \left(\int_M e^{- p u} \chi^{l - 1}_{t u} \wedge \omega^{n - l + 1} \right) dt  + C \int_M e^{- p u} \omega^n.
\end{aligned}
\end{equation}

By the concavity of the quotient functions, we have
\begin{equation}
\label{quotient-lemma-estimate-left-1}
	k \chi^{k - 1}_{t u} \wedge \omega^{n - k} - l \psi \chi^{l - 1}_{t u} \wedge \omega^{n - l} > 0 .
\end{equation}
Moreover, for some $\delta > 0$,
\begin{equation}
\label{quotient-lemma-estimate-left-2}
\begin{aligned}
	&\, \sqrt{-1} \p u \wedge \bpartial u \wedge \left(k \chi^{k - 1}_{t u} \wedge \omega^{n - k} - l \psi \chi^{l - 1}_{t u} \wedge \omega^{n - l}\right) \\
	\geq&\, \left(\left( (1 - t) (1 + \delta) + t \right)^{k - l} - 1\right) l \psi \sqrt{-1} \p u \wedge \bpartial u \wedge \chi^{l - 1}_{t u} \wedge \omega^{n - l} \\
	\geq&\, (k - l) (1 - t) \delta l \psi \sqrt{- 1} \p u \wedge \bpartial u \wedge \chi^{l - 1}_{t u} \wedge \omega^{n - l} . 
\end{aligned}
\end{equation}
Applying \eqref{hessian-lemma-estimate-3}, there are uniform constants $c_1 > 0$ and $c_2 > 0$ such that
\begin{equation}
\label{quotient-lemma-estimate-left-3}
\begin{aligned}
	&\, \int^1_0 \left(\int_M e^{- p u} \sqrt{- 1} \p u \wedge \bpartial u \wedge \left(k \chi^{k - 1}_{t u} \wedge \omega^{n - k} - l \psi \chi^{l - 1}_{t u} \wedge \omega^{n - l}\right)\right) dt \\
	\geq&\,  (k - l) \delta l \int^1_0 (1 - t)^l \left(\int_M \psi  e^{- p u} \sqrt{- 1} \p u \wedge \bpartial u \wedge \chi^{l - 1} \wedge \omega^{n - l} \right)dt \\
	\geq&\, c_1 \int_M e^{- p u} \sqrt{-1} \p u \wedge \bpartial u \wedge \chi^{l - 1} \wedge \omega^{n - l} 
\end{aligned}
\end{equation}
and for $0 < t < \frac{1}{2}$
\begin{equation}
\label{quotient-lemma-estimate-left-4}
	k \chi^{k - 1}_{t u} \wedge \omega^{n - k} - l \psi \chi^{l - 1}_{t u} \wedge \omega^{n - l} > c_2 \chi^{l - 1}_{t u} \wedge \omega^{n - l} .
\end{equation}
According to \eqref{quotient-lemma-estimate-left-3}, the left term in \eqref{quotient-lemma-estimate-10} is positive, and we will use part of it to deal with the first term in the right side of \eqref{quotient-lemma-estimate-10}. We compute
\begin{equation}
\label{quotient-lemma-estimate-right-1}
\begin{aligned}
	&\, \frac{C}{p} \int^{\frac{1}{2}}_0 \left(\int_M e^{- p u} \chi^{l - 1}_{t u} \wedge \omega^{n - l + 1}\right) dt \\
	=&\, {C(l - 1)} \int^{\frac{1}{2}}_0 \int^t_0 \left(\int_M e^{- p u} \sqrt{-1} \p u \wedge \bpartial u \wedge \chi^{l - 2}_{s u} \wedge \omega^{n - l + 1}\right) ds\, dt \\
	&\, + \frac{C}{2 p} \int_M e^{- p u} \chi^{l - 1} \wedge \omega^{n - l + 1} \\
	\leq&\, \frac{C(l - 1)}{2} \int^{\frac{1}{2}}_0 \left(\int_M e^{- p u} \sqrt{-1} \p u \wedge \bpartial u \wedge \chi^{l - 2}_{t u} \wedge \omega^{n - l + 1}\right) dt \\
	&\, + \frac{C}{2 p} \int_M e^{- p u} \chi^{l - 1} \wedge \omega^{n - l + 1} .
\end{aligned}
\end{equation}
The second term is acceptable, and we only need to control the first term. Notice that there is a uniform positive constant $\lambda$ such that $\chi - \lambda \omega \in \Gamma^k_\omega$. We compute, using integration by parts and G\r{a}ding's inequality,
\begin{equation}
\label{quotient-lemma-estimate-right-2}
\begin{aligned}
	&\, \int^{\frac{1}{2}}_0 \left(\int_M e^{- p u} \sqrt{-1} \p u \wedge \bpartial u \wedge \chi^{l - 1}_{t u} \wedge \omega^{n - l}\right) dt \\
	\geq&\, \lambda\int^{\frac{1}{2}}_0 \left(\int_M e^{- p u} \sqrt{-1} \p u \wedge \bpartial u \wedge \chi^{l - 2}_{t u} \wedge \omega^{n - l + 1}\right) dt \\
	&\, + \frac{1}{l - 1} \int^{\frac{1}{2}}_0 \left(\int_M e^{- p u} \sqrt{-1} \p u \wedge \bpartial u \wedge t \frac{d}{d t} \chi^{l - 1}_{t u} \wedge \omega^{n - l }\right) dt \\
	\geq&\, \lambda\int^{\frac{1}{2}}_0 \left(\int_M e^{- p u} \sqrt{-1} \p u \wedge \bpartial u \wedge \chi^{l - 2}_{t u} \wedge \omega^{n - l + 1}\right) dt \\
	&\, - \frac{1}{l - 1} \int^{\frac{1}{2}}_0 \left(\int_M e^{- p u} \sqrt{- 1} \p u \wedge \bpartial u \wedge \chi^{l - 1}_{t u} \wedge \omega^{n - l} \right) dt
\end{aligned}
\end{equation}
and consequently
\begin{equation}
\label{quotient-lemma-estimate-right-3}
\begin{aligned}
	&\, \frac{l}{l - 1} \int^{\frac{1}{2}}_0 \left(\int_M e^{- p u} \sqrt{- 1} \p u \wedge\bpartial u \wedge \chi^{l - 1}_{t u} \wedge \omega^{n - l} \right) dt \\
	\geq&\, \lambda \int^{\frac{1}{2}}_0 \left(\int_M e^{- p u} \sqrt{- 1} \p u \wedge \bpartial u \wedge \chi^{l - 2}_{t u} \wedge \omega^{n - l + 1} \right) dt .
\end{aligned}
\end{equation}
Combining \eqref{quotient-lemma-estimate-10}, \eqref{quotient-lemma-estimate-left-3}, \eqref{quotient-lemma-estimate-left-4}, \eqref{quotient-lemma-estimate-right-1} and \eqref{quotient-lemma-estimate-right-3}, we may assume that $p_0 \geq \frac{l C}{c_2}$ and thus for $p \geq p_0$,
\begin{equation}
\begin{aligned}
	\frac{c_1 p \lambda^{l - 1}}{2} \int_M e^{- p u} \sqrt{-1} \p u \wedge \bpartial u \wedge \omega^{n - 1}  \leq C \int_M e^{- p u} \omega^n .
\end{aligned}
\end{equation}

\end{proof}

\bigskip

\section{The estimates for complex Monge-Amp\`ere type equations}
\label{hermitian}
\setcounter{equation}{0}
\medskip

For general Hermitian manifolds, there are troublesome torsion terms which are extremely difficult to deal with. We shall focus on the complex Monge-Amp\`ere type equations, as the positivity of $\Gamma^n_\omega$ and $S_n/S_{n - \a}$ does help to control the torsion terms (see also \cite{TWv10b}, \cite{GSun12}, \cite{Sun2013e}).

The gradient estimate and $C^2$ estimate were obtained by Guan and the author in \cite{GSun12}, while a sharp form of $C^2$ estimate was also achieved by the author in \cite{Sun2013e}. Higher order estimates are guaranteed by Evans-Krylov theory and Schauder estimates, which is standard. Therefore, it suffices to obtain a uniform $L^\infty$ bound of $u$.
\begin{lemma} 
\label{inverse-hessian-lemma-estimate}
Let $u$ be a smooth admissible solution to complex Monge-Amp\`ere type equation~\eqref{ma2-main-equation}. Then there are uniform constants $C$, $p_0$ such that for all $p \geq p_0$ we have inequality~\eqref{ma2_2_main_inequality}.
\end{lemma}
\begin{proof}

We follow the proof of Lemma~\ref{quotient-lemma-estimate} with some modification to deal with the torsion terms. Without loss of generality, we may assume
\begin{equation}
\label{initial-cone-condition}
	n \chi^{n - 1} > (n - \a) \psi \chi^{n - \a - 1} \wedge \omega^\a ,
\end{equation}
and there exist uniform positive constants $\lambda$ and $\Lambda$.
\begin{equation}
\label{uniform-positivity}
	\lambda \omega \leq \chi \leq \Lambda \omega .
\end{equation}
Also, by the monotony of $S_n/S_{n - \a}$, we have
\begin{equation}
\label{final-cone-condition}
n \chi^{n - 1}_u > (n - \a) \psi \chi^{n - \a - 1}_u \wedge \omega^\a .
\end{equation}

As in the proof of Lemma~\ref{quotient-lemma-estimate}, we study the integral
\begin{equation}
\label{target-integral}
	I : = \int_M e^{- p u} \left((\chi^n_u - \chi^n) - \psi (\chi^{n - \a}_u \wedge \omega^\a - \chi^{n - \a} \wedge \omega^\a)\right) .
\end{equation}
On one hand,
\begin{equation}
\label{integral-inequality-less}
	I = \int_M e^{- p u} \left(\frac{\chi^n_u}{\chi^{n - \a}_u \wedge \omega^\a} - \frac{\chi^n}{\chi^{n - \a} \wedge \omega^\a}\right) \chi^{n - \a} \wedge \omega^\a \leq C \int_M e^{- p u} \omega^n .
\end{equation}
On the other hand, we have
\begin{equation}
\label{integral-inequality-greater}
\begin{aligned}
	I &= p \int^1_0 \left( \int_M e^{- p u} \sqrt{- 1} \p u \wedge \bpartial u \wedge  \left(n \chi^{n - 1}_{t u} - (n - \a) \psi \chi^{n - \a - 1}_{t u} \wedge \omega^\a \right)\right) dt \\
	&\quad - \frac{1}{p} \int^1_0 \left(\int_M \sqrt{- 1} e^{- p u} \p\bpartial \left(n \chi^{n - 1}_{t u} - (n - \a) \psi \chi^{n - \a - 1}_{t u} \wedge \omega^\a \right)\right) dt \\
	&\geq p \int^1_0 \left( \int_M e^{- p u} \sqrt{- 1} \p u \wedge \bpartial u \wedge  \left(n \chi^{n - 1}_{t u} - (n - \a) \psi \chi^{n - \a - 1}_{t u} \wedge \omega^\a \right)\right) dt \\
	&\quad - \frac{C}{p} \int^1_0 \left(\int_M e^{- p u} \left( \chi^{n - 2}_{t u} \wedge \omega^2 + \chi^{n - 3}_{t u}\wedge\omega^3 + \chi^{n - \a - 1}_{t u} \wedge \omega^{\a + 1} \right)\right) dt \\
	&\quad - \frac{C}{p} \int^1_0 \left(\int_M e^{- p u} \left(  \chi^{n - \a - 2}_{t u}  \wedge \omega^{\a + 2} + \chi^{n - \a - 3}_{t u} \wedge \omega^{\a + 3} \right)\right) dt ,
\end{aligned}
\end{equation}
where the inequality uses the uniform bounds in \eqref{uniform-positivity}.

Using integration by parts and G\r{a}ding's inequality, for $n \geq k \geq 1$,
\begin{equation}
\label{metric-index-inequality}
\begin{aligned}
	\int^1_0 \chi^k_{t u} \wedge \omega^{n - k} dt	&\geq \lambda \int^1_0 \chi^{k - 1}_{t u} \wedge \omega^{n - k + 1} dt + \frac{1}{k} \int^1_0 t \frac{d}{d t} (\chi^{k}_{t u} \wedge \omega^{n - k}) dt \\
	&= \lambda \int^1_0 \chi^{k - 1}_{t u} \wedge \omega^{n - k + 1} dt+ \frac{1}{k} \chi^k_u \wedge \omega^{n - k} - \frac{1}{k} \int^1_0 \chi^k_{t u} \wedge \omega^{n - k} dt
\end{aligned}
\end{equation}
and hence
\begin{equation}
\label{metric-index-inequality-1}
	\frac{k + 1}{k} \int^1_0 \chi^k_{t u} \wedge \omega^{n - k} dt > \lambda \int^1_0 \chi^{k - 1}_{t u} \wedge \omega^{n - k + 1} dt .
\end{equation}
Combining   \eqref{integral-inequality-less}, \eqref{integral-inequality-greater}, \eqref{metric-index-inequality-1} and \eqref{quotient-lemma-estimate-8}, 
\begin{equation}
\label{integral-inequality}
\begin{aligned}
	& p \int^1_0 \left( \int_M e^{- p u} \sqrt{- 1} \p u \wedge \bpartial u \wedge  \left(n \chi^{n - 1}_{t u} - (n - \a) \psi \chi^{n - \a - 1}_{t u} \wedge \omega^\a \right)\right) dt \\
	&\qquad \leq  C \int_M e^{- p u} \omega^n +  \frac{C}{p} \int^\frac{1}{2}_0 \left(\int_M e^{- p u} \chi^{n - 2}_{t u} \wedge \omega^2 \right) dt .
\end{aligned}
\end{equation}

By the concavity of the quotient functions, for some $\delta > 0$,
\begin{equation}
\label{quotient-inequality-2}
	n \chi^{n - 1}_{t u} - (n - \a) \psi \chi^{n - \a - 1}_{t u} \wedge \omega^\a > n \left(1 - \frac{1}{\left(1 + \delta - \delta t\right)^\a}\right) \chi^{n - 1}_{t u}.
\end{equation}
Consequently, for some uniform constants $c_1 > 0$, 
\begin{equation}
\label{metric-inequality-2}
\begin{aligned}
	&\,  \int^1_0 \left( \int_M e^{- p u} \sqrt{- 1} \p u \wedge \bpartial u \wedge  \left(n \chi^{n - 1}_{t u} - (n - \a) \psi \chi^{n - \a - 1}_{t u} \wedge \omega^\a \right)\right) dt \\
	\geq&\, c_1 \int_M e^{- p u} \sqrt{- 1} \p u \wedge \bpartial u \wedge \omega^{n - 1},
\end{aligned}
\end{equation}
and for some uniform constant $c_2 > 0$ and $0 \leq t \leq \frac{1}{2}$
\begin{equation}
\label{metric-inequality-3}
	n \chi^{n - 1}_{t u} - (n - \a) \psi \chi^{n - \a - 1}_{t u} \wedge \omega^\a > c_2 \chi^{n - 1}_{t u}.
\end{equation}
Applying \eqref{metric-index-inequality-1} and assuming that $p$ is large enough,
\begin{equation}
\label{metric-inequality-4}
\begin{aligned}
	&\quad\; \frac{1}{p} \int^\frac{1}{2}_0 \left(\int_M e^{- p u} \chi^{n - 2}_{t u} \wedge \omega^2 \right) dt \\
	&= \frac{(n - 2)}{p} \int^\frac{1}{2}_0 \int^t_0 \left(\int_M e^{- p u} \sqrt{- 1} \p\bpartial u \wedge \chi^{n - 3}_{s u} \wedge \omega^2 \right)ds dt + \frac{1}{2 p} \int_M e^{- p u} \chi^{n - 2} \wedge \omega^2 \\
	&= (n - 2) \int^\frac{1}{2}_0 \int^t_0 \left(\int_M e^{- p u} \sqrt{- 1} \p u \wedge \bpartial u \wedge \chi^{n - 3}_{s u} \wedge \omega^2\right)ds dt \\
	&\quad + \frac{(n - 2)}{p^2} \int^\frac{1}{2}_0 \int^t_0 \left(\int_M e^{- p u}  \sqrt{- 1} \p \bpartial (\chi^{n - 3}_{s u} \wedge \omega^2)\right)ds dt + \frac{1}{2 p} \int_M e^{- p u} \chi^{n - 2} \wedge \omega^2 \\
	&\leq \frac{(n - 2)}{2} \int^\frac{1}{2}_0 \left(\int_M e^{- p u} \sqrt{- 1} \p u \wedge \bpartial u \wedge \chi^{n - 3}_{t u} \wedge \omega^2\right) dt \\
	&\quad + \frac{1}{2 p} \int^\frac{1}{2}_0 \left(\int_M e^{- p u}  \chi^{n - 2}_{t u} \wedge \omega^2 \right) dt  + \frac{1}{2 p} \int_M e^{- p u} \chi^{n - 2} \wedge \omega^2,
\end{aligned}
\end{equation}
and thus
\begin{equation}
\label{metric-inequality-5}
\begin{aligned}
	&\qquad\frac{1}{ p} \int^\frac{1}{2}_0 \left(\int_M e^{- p u}  \chi^{n - 2}_{t u} \wedge \omega^2 \right) dt   \\
	&\leq (n - 2) \int^\frac{1}{2}_0 \left(\int_M e^{- p u} \sqrt{- 1} \p u \wedge \bpartial u \wedge \chi^{n - 3}_{t u} \wedge \omega^2\right) dt  + \frac{1}{p} \int_M e^{- p u} \chi^{n - 2} \wedge \omega^2 .
\end{aligned}
\end{equation}

Combining \eqref{integral-inequality}, \eqref{metric-inequality-2}, \eqref{metric-inequality-3}, \eqref{metric-inequality-5} and \eqref{quotient-lemma-estimate-right-3}, there are uniform constants C and $p_0$  such that for $p \geq p_0$, 
\begin{equation}
	p \int_M e^{- p u} \sqrt{-1} \p u \wedge \bpartial u \wedge \omega^{n - 1} \leq C \int_M e^{- p u} \omega^n .
\end{equation}

\end{proof}

\bigskip
\noindent
{\bf Acknowledgements}\quad   
The author is very grateful to Bo Guan for his support and encouragement. The author also wishes to thank Valentino Tosatti and Ben Weinkove for some helpful discussions.

\end{document}